\documentclass[draft]{amsart}
\usepackage{amssymb} 
\newtheorem{theorem}{Theorem}[section]
\newtheorem{lemma}[theorem]{Lemma}
\newtheorem{proposition}[theorem]{Proposition}
\newtheorem{corollary}[theorem]{Corollary}

\newtheorem{question}[theorem]{Question}

\theoremstyle{definition}
\newtheorem{definition}[theorem]{Definition}

\newtheorem{example}[theorem]{Example}

\newtheorem{remark}[theorem]{Remark}
\theoremstyle{approach}

\numberwithin{equation}{section}

\begin{document}
\setcounter{page}{1}
\title[On the Arens regularity of Fr\'echet algebras and their biduals]{On the Arens regularity of Fr\'echet algebras and their biduals}
\author[Z. Alimohammadi and A. Rejali]{Zahra Alimohammadi and Ali Rejali}


\subjclass[2010]{Primary 46H05; Secondary 46A04, 43A60.} \keywords{Arens regular, locally convex space, Fr\'echet algebra, weakly almost periodic function}

\begin{abstract}
In this paper, we study the concept of weakly almost periodic functions on Fr\'echet algebras. 
For a Fr\'echet algebra $\mathcal{A}$, we show that 
$WAP(\mathcal{A})=wap(\mathcal{A})$. We also show that $\mathcal{A}^{**}$ is  Arens regular if and only if both $\mathcal{A}$ and $WAP(\mathcal{A})^*$ are Arens regular. 
Finally, for a sequence of Fr\'echet algebras $(\mathcal{A}_n)$, we prove that the Fr\'echet algebra $\ell^1\text{-}\prod_{n\in\mathbb{N}}\mathcal{A}_n$ is Arens regualr if and only if each $\mathcal{A}_n$ is Arens regular.
\end{abstract}

\maketitle \setcounter{section}{0}

\section{\bf Introduction and preliminaries}\label{S0S}

Let $G$ be a locally compact Abelian group, and let $C_b(G)$ denote the Banach algebra of all bounded continuous complex-valued functions on $G$. Due to Eberlein \cite{Eberlein}, $f$ in $C_b(G)$ is weakly almost periodic if the set $\{s\cdot f: s\in G\overline{\}}^w$ is weakly compact, where $s\cdot f(t)=f(ts)$ for $s,t\in G$. Given a topological space $E$, Grothendieck 
showed that a set $A\subseteq C_b(E)$ is weakly relatively compact if and only if it is bounded and it is impossible to choose sequences $(f_m)$ in $A$ and $(a_n)$ in $E$ such that $\lim_m\lim_nf_m(a_n)$ and $\lim_n\lim_mf_m(a_n)$ both exist and are distinct. Therefore, $f$ is weakly almost periodic if and only if $\lim_m\lim_nf(a_mb_n)$ and $\lim_n\lim_mf(a_mb_n)$ are equal whenever they both exist \cite[Theorem 6]{Grothendiek}. Similarly, Young proved this theorem for $\ell^{\infty}(S)$, where 
$S$ is a semigroup \cite{Young}.
Later, Baker and Rejali \cite{Baker} studied the Arens regularity of $\ell^1(S,w)$, where $w$ is a weight function on the discrete semigroup $S$. 
A survey of results related to weighted group algebras is given by Dales and Lau \cite{Dales}.
Also, Duncan and Hosseiniun provided a survey paper about Arens regularity of Banach algebras \cite{Duncan}. Arikan \cite{Arikan} showed that for a sequence of Banach algebras $(\mathcal{A}_n)$ , $\ell^1\text{-}\prod_{n\in\mathbb{N}}\mathcal{A}_n$ is Arens regular if and only if each $\mathcal{A}_n$ is Arens regular. 
Let $\mathcal{A}^*$ and $\mathcal{A}^{**}$ denote the dual and bidual of a Banach algebra $\mathcal{A}$, respectively. Pym \cite{Pym} showed that $f\in\mathcal{A}^*$ is weakly almost periodic if and only if 
$$
F\square G(f)=F\Diamond G(f)\;\;\;\;\;\;\;(F,G\in\mathcal{A}^{**}).
$$
Gulick \cite{Gulick} studied Arens products on the bidual of locally multiplicatively-convex algebras and especially Fr\'echet algebras. The Arens regularity of Fr\'echet algebras was defined by Zivari-Kazempour \cite{Zivari}. Indeed, a Fr\'echet algebra $\mathcal{A}$ is called Arens regular if the products 
$\square$ and $\Diamond$ coincide on $\mathcal{A}^{**}$.
In this paper, we generalize Pym's criteria to the Fr\'echet case. We show that $\mathcal{A}$ is Arens regular if and only if $WAP(\mathcal{A})=\mathcal{A}^*$, where $WAP(\mathcal{A})$ is the space of all weakly almost periodic functions on $\mathcal{A}$.

Before proceeding to the main results, we provide some basic definitions and frameworks
related to locally convex spaces which will be required throughout the paper. See
\cite{Goldmann}, \cite{Hel2}, \cite{Meise}, and \cite{Sch} for
more information in this field. 

A locally convex space $E$ is a topological vector space in which each point has a neighborhood basis of convex sets. Throughout the paper, all locally convex spaces are
assumed to be Hausdorff. For a locally convex space $E$, a collection $\mathcal{U}$ of zero neighborhoods in $E$ is called a fundamental system of zero neighbourhoods, if for every zero neighborhood $U$ there exists a $V\in\mathcal{U}$ and an $\varepsilon>0$ with $\varepsilon V\subseteq U$. A family $(p_{\alpha})_{\alpha\in
\Lambda}$ of continuous seminorms on $E$ is called a fundamental system of seminorms, if the sets 
$U_{\alpha}:=\{a\in E: p_{\alpha}(a)<1\}$ $(\alpha\in\Lambda)$
form a fundamental system of zero neighborhoods \cite[page 251]{Meise}.
By \cite[Lemma 22.4]{Meise}, every locally convex Hausdorff space $E$
has a fundamental system of seminorms $(p_{\alpha})_{\alpha\in
\Lambda}$ and we denote it by $(E,p_{\alpha})$. 
For locally convex spaces $E$ and
$F$, the bounded operator $T\in\mathcal{L}(E,F)$ is called weakly compact if there exists a zero neighborhood $U$ in $E$ such that $\overline{T(U)}^w$ is weakly compact in F \cite[Page 204]{Kothe2}.

A complete metrizable locally convex
space is called a Fr\'echet space. In fact, a Fr\'echet space $\mathcal{A}$ is a
locally convex space which has a countable fundamental system of
seminorms $(p_{\ell})_{\ell\in\Bbb N}$. 
Following \cite{Meise}, for a Fr\'echet space $\mathcal{A}$, the strong (resp. weak$^*$) topology on $\mathcal{A}^*$ is created by the seminorm system $(q_M)$, defined by
$$
q_M(f):=\sup_{a\in M}|f(a)|\;\;\;\;\;\;\;(f\in\mathcal{A}^*),
$$
where $M\subseteq\mathcal{A}$ is bounded (resp. finite). Similarly, the strong (resp. weak$^*$) topology on $\mathcal{A}^{**}$ is defined by bounded (resp. finite) subsets of $\mathcal{A}^*$. 

A topological algebra ${\mathcal A}$ is an algebra, which is a
topological vector space and the multiplication ${\mathcal
A}\times{\mathcal A}\to {\mathcal A}$ $((a,b)\mapsto
ab)$ is separately continuous. A subset $U$ of ${\mathcal A}$ is idempotent if $U^2\subseteq U$. Moreover, the polar of $U$ is defined as 
$$U^{\circ}=\{f\in\mathcal{A}^*: |f(a)|\leq 1\; \text{for all}\; a\in U\}.$$ 

A complete topological
algebra ${\mathcal A}$ is a Fr\'echet algebra if its topology is produced by a
countable family of increasing submultiplicative seminorms; see \cite{Goldmann}. 
In section \ref{section2}, we show that the dual of $WAP(\mathcal{A})$ is a Fr\'echet algebra under the convolution and evolution products. In section \ref{section3}, we study Grothendieck and Pym theorems for Fr\'echet algebras. In particular, we show that $f\in\mathcal{A}^*$ is weakly almost periodic if and only if there exists a zero neighborhood $U$ in $\mathcal{A}$ such that 
$$\big{\{}a\cdot (f|_{U})^{\beta}: a\in U\overline{\big{\}}}^w$$ 
is weakly compact in $C(\beta U)$, where $\beta U$ is the Stone-\u Cech compactification of $U$.
Furthermore, $f\in WAP(\mathcal{A})$ if and only if for sequences $(a_m)$ and $(b_n)$ in $U$ with distinct elements
$$\lim_m\lim_nf(a_mb_n)=\lim_n\lim_mf(a_mb_n),$$
whenever both limits exist. Moreover, $\mathcal{A}^{**}$ is Arens regular if and only if both $\mathcal{A}$ and $WAP(\mathcal{A})^*$ are Arens regular. Finally, in section \ref{section4}, we show that for a sequence of Fr\'echet algebras $(\mathcal{A}_n,p^n_{\ell})$,
$\ell^1\text{-}\prod_{n\in\mathbb{N}}\mathcal{A}_n$ is a Fr\'echet algebra. Also, $\ell^1\text{-}\prod_{n\in\mathbb{N}}\mathcal{A}_n$ is Arens regular if and only if each $\mathcal{A}_n$ is Arens regular.

\section{\bf Weakly almost periodic functions}\label{section2}

Let $(\mathcal{A},p_{\ell})$ be a Fr\'echet algebra and $f\in\mathcal{A}^*$. We say that $f\in WAP(\mathcal{A})$ if and only if the map $T_f:\mathcal{A}\to\mathcal{A}^*$ defined by $a\mapsto a\cdot f$
is weakly compact. In other words, $f\in WAP(\mathcal{A})$ if and only if there exists a zero neighborhood $U$ in $\mathcal{A}$ such that $\{a\cdot f: a\in U\overline{\}}^w$ is weakly compact in $\mathcal{A}^*$.

In the above definition, we can choose an idempotent neighborhood $U$ of $0$, since the
seminorms $(p_{\ell})_{\ell\in\mathbb{N}}$ are submultiplicative and create a fundamental system; see \cite[page 251]{Meise}. 

\begin{lemma}\label{111llll1}
Let $(\mathcal{A},p_{\ell})$ be a Fr\'echet algebra.
For all $a\in\mathcal{A}$ and $f\in WAP(\mathcal{A})$, $f\cdot a$ and $a\cdot f$ belong to $WAP(\mathcal{A})$.
\end{lemma}

\begin{proof}
Let $f\in WAP(\mathcal{A})$ and $U$ be an idempotent neighborhood of $0$ such that the set $\{a\cdot f: a\in U\overline{\}}^w$ is weakly compact in $\mathcal{A}^*$. Clearly, $a\cdot f\in WAP(\mathcal{A})$, for every $a\in U$. By using \cite[page 251]{Meise}, there are $\ell\in\mathbb{N}$ and $\varepsilon>0$ such that $\{a\in\mathcal{A}: p_{\ell}(a)<\varepsilon\}\subseteq U$. Thus, for all $a\in\mathcal{A}$ we have 
$\frac{\varepsilon a}{p_{\ell}(a)+1}\in U$ and consequently $\frac{\varepsilon a}{p_{\ell}(a)+1}\cdot f\in WAP(\mathcal{A})$. Therefore, the set 
$\{b\cdot (a\cdot f):\;b\in U\overline{\}}^w$ is weakly compact in $\mathcal{A}^*$ and $a\cdot f\in WAP(\mathcal{A})$. Similarly, 
$f\cdot a\in WAP(\mathcal{A})$, for every $a\in\mathcal{A}$.
\end{proof}

Similar to the definition of Arens products on $\mathcal{A}^{**}$, we can define the convolution and evolution products on $WAP(\mathcal{A})^*$. 
Indeed, for $f\in WAP(\mathcal{A})$ and $F\in WAP(\mathcal{A})^*$ define
$$
F\cdot f(a):=F(f\cdot a)\;\;\text{and}\;\; f\cdot F(a):=F(a\cdot f)\;\;\;\;\;\;(a\in \mathcal{A}).
$$
Also, for $F,G\in WAP(\mathcal{A})^*$, define
$$F\ast G(f):=F(G\cdot f)\;\;\text{and}\;\; F\circ G(f):=G(f\cdot F)\;\;\;\;\;\;(f\in WAP(\mathcal{A})).$$
In the following lemma we show that $F\cdot f$ and $f\cdot F$ belong to $WAP(\mathcal{A})$, and $F\ast G$ and $F\circ G$ belong to $WAP(\mathcal{A})^*$.

\begin{lemma}\label{lGu11}
Let $(\mathcal{A},p_{\ell})$ be a Fr\'echet algebra. The following assertions hold.
\begin{enumerate}
\item[(i)]
For all $f\in WAP(\mathcal{A})$ and $F\in WAP(\mathcal{A})^*$, $F\cdot f$ and $f\cdot F$ belong to $WAP(\mathcal{A})$.
\item[(ii)]
For all $F,G\in WAP(\mathcal{A})^*$, the convolution and evolution products $F\ast G$ and $F\circ G$, respectively, belong to $WAP(\mathcal{A})^*$. 
\end{enumerate}
\end{lemma}

\begin{proof}
(i).
Clearly, for $a\in \mathcal{A}$, $f\in WAP(\mathcal{A})$, and $F\in \mathcal{A}^{**}$, we have $f\cdot F\in \mathcal{A}^*$ and $a\cdot (f\cdot F)=(a\cdot f)\cdot F$.
By assumption, let $U$ be a zero neighborhood in $\mathcal{A}$ such that
$\{a\cdot f: a\in U\overline{\}}^w$ is weakly compact in $\mathcal{A}^*$. For each net $(a_{\alpha})\subseteq U$ there exists a subnet $(a_{\beta})$ such that $(a_{\beta}\cdot f)$ is weakly convergent to some $g$ in $\{a\cdot f: a\in U\overline{\}}^w$. Thus, for every $G\in\mathcal{A}^{**}$, we have
$$
G((a_{\beta}\cdot f)\cdot F)=F\Diamond G(a_{\beta}\cdot f)\longrightarrow_{\beta} F\Diamond G(g)=G(g\cdot F).
$$
Therefore, $a_\beta\cdot (f\cdot F)\stackrel{w}{\longrightarrow}g\cdot F$ and consequently $g\cdot F\in \{a\cdot (f\cdot F): a\in U\overline{\}}^w$. Hence, $\{a\cdot (f\cdot F): a\in U\overline{\}}^w$ is weakly compace in $\mathcal{A}^*$ and $f\cdot F\in WAP(\mathcal{A})$. Now, let $f\in WAP(\mathcal{A})$ and $F\in WAP(\mathcal{A})^*$. 
By applying \cite[Proposition 22.12]{Meise}, there exists $\bar{F}\in \mathcal{A}^{**}$ such that $\bar{F}|_{WAP(\mathcal{A})}=F$. Thus, by Lemma \ref{111llll1}, $f\cdot\bar{F}(a)=f\cdot F(a)$ for every $a\in\mathcal{A}$. Since $f\cdot\bar{F}\in WAP(\mathcal{A})$, we have $f\cdot F\in WAP(\mathcal{A})$. Similarly, $F\cdot f\in WAP(\mathcal{A})$.

(ii).
It follows in the same way as \cite[Lemma 3.4]{Gulick}.
\end{proof}

We now introduce a countable system of increasing submultiplicative seminorms on $\mathcal{A}^{**}$. Let $(\mathcal{A},p_{\ell})$ be a Fr\'echet algebra and $M_{\ell}=\{a\in\mathcal{A}: p_{\ell}(a)<1\}$ ($\ell\in\mathbb{N}$). For $f\in\mathcal{A}^*$, define $q_{M_{\ell}}(f)=\sup_{a\in M_{\ell}} |f(a)|$ ($\ell\in\mathbb{N}$). For each $\ell\in\mathbb{N}$, by Alaoglu-Bourbaki theorem \cite[Theorem 23.5]{Meise}, $M^{\circ}_{\ell}$ is weak$^*$-compact and so is weak$^*$-bounded in $\mathcal{A}^*$; see \cite[Remark 23.6]{Meise}. In Lemma \ref{fundamentalllll}, we show that the seminorms $(q_{\ell})_{\ell\in\mathbb{N}}$, defined by $q_{\ell}(F)=\sup_{f\in M^{\circ}_{\ell}} |F(f)|$ $(F\in\mathcal{A}^{**})$, create a fundamental system of seminorms on $\mathcal{A}^{**}$.

\begin{lemma}\label{plllqlllfunde}
If $(\mathcal{A},p_{\ell})$ is a Fr\'echet algebra and $f\in\mathcal{A}^*$, then the following are valid.
\begin{enumerate}
\item[(i)]
There exist $\ell\in\mathbb{N}$ and $\lambda>0$ such that
$|f(a)|\leq \lambda p_{\ell}(a)$ $(a\in\mathcal{A})$.
\item[(ii)]
There exists an $\ell\in\mathbb{N}$ such that $q_{M_{\ell}}(f)<\infty$.
\item[(iii)]
For all $a\in\mathcal{A}$ and $\ell\in\mathbb{N}$ we have $|f(a)|\leq q_{M_{\ell}}(f)p_{\ell}(a)$.
\item[(iv)]
For all $F\in\mathcal{A}^{**}$ and $\ell\in\mathbb{N}$ we have $|F(f)|\leq q_{\ell}(F) q_{M_{\ell}}(f)$.
\end{enumerate}
\end{lemma}

\begin{proof}
To prove (i), let $U:=\{z\in\mathbb{C}: |z|<1\}$. By continuity of $f$, there exists a zero neighborhood $V$ in $\mathcal{A}$ such that $f(V)\subseteq U$. Now, by \cite[page 251]{Meise}, there are $\ell\in\mathbb{N}$ and $\delta>0$ such that $\{a\in\mathcal{A}: p_{\ell}(a)<\delta\}\subseteq V$. Let $a\in\mathcal{A}$ and consider the arbitrary real number $\varepsilon>0$.
Since $p_{\ell}(\frac{\delta a}{2(p_{\ell}(a)+\varepsilon)})<\delta$, we have 
$|f(a)|<\frac{2}{\delta}(p_{\ell}(a)+\varepsilon)$.
By setting $\lambda:=\frac{2}{\delta}$, we have $|f(a)|\leq\lambda p_{\ell}(a)$, which completes the proof.

(ii) is an immediate consequence of (i).

To show (iii) and (iv), consider the arbitrary real number $\varepsilon>0$. Clearly,
$$
|f(\frac{a}{p_{\ell}(a)+\varepsilon})|\leq q_{M_{\ell}}(f)\;\;\;\;\;\;(a\in\mathcal{A},\;\ell\in\mathbb{N}),
$$
and (iii) holds.
Now, it is easy to see that $\frac{f}{q_{M_{\ell}}(f)+\varepsilon}\in M^{\circ}_{\ell}$ ($\ell\in\mathbb{N}$).
Thus,
$$
|F(\frac{f}{q_{M_{\ell}}(f)+\varepsilon})|\leq q_{\ell}(F)\;\;\;\;\;\;(F\in\mathcal{A}^{**},\;\ell\in\mathbb{N}),
$$
and (iv) holds.
\end{proof}

\begin{lemma}\label{fundamentalllll}
Let $(\mathcal{A},p_{\ell})$ be a Fr\'echet algebra. Then $(q_{\ell})_{\ell\in\mathbb{N}}$ is a fundamental system of submultiplicative seminorms on $\mathcal{A}^{**}$.
\end{lemma}

\begin{proof}

It is enough to prove that the seminorms $(q_{\ell})_{\ell\in\mathbb{N}}$ are submultiplicative. By Lemma \ref{plllqlllfunde} (iii), for $f\in\mathcal{A}^*$ we have
$$
|f(ab)|\leq q_{M_{\ell}}(f)p_{\ell}(a)p_{\ell}(b)\;\;\;\;\;\;(a,b\in\mathcal{A},\;\ell\in\mathbb{N}),
$$
and
$$
q_{M_{\ell}}(f\cdot a)=\sup_{b\in M_{\ell}}|f(ab)|\leq q_{M_{\ell}}(f)p_{\ell}(a)\;\;\;\;\;\;(a\in\mathcal{A},\;\ell\in\mathbb{N}).
$$
Thus, by Lemma \ref{plllqlllfunde} (iv), for $G\in\mathcal{A}^{**}$ we have
$$
|G(f\cdot a)|\leq q_{\ell}(G) q_{M_{\ell}}(f)p_{\ell}(a)\;\;\;\;\;\;(a\in\mathcal{A},\;f\in\mathcal{A}^*,\;\ell\in\mathbb{N}),
$$
and 
$$
q_{M_{\ell}}(G\cdot f)\leq q_{\ell}(G) q_{M_{\ell}}(f)\;\;\;\;\;\;(f\in\mathcal{A}^*,\;\ell\in\mathbb{N}).
$$
Therefore, for $\ell\in\mathbb{N}$, $f\in\mathcal{A}^*$, and $F,G\in\mathcal{A}^{**}$, we have
$$
|F\square G(f)|=|F(G\cdot f)|\leq q_{\ell}(F)q_{M_{\ell}}(G\cdot f)\leq q_{\ell}(F)q_{\ell}(G) q_{M_{\ell}}(f),
$$
and
$$
q_{\ell}(F\square G)=\sup_{f\in M^{\circ}_{\ell}}|F\square G(f)|\leq q_{\ell}(F)q_{\ell}(G).
$$
Similar to the previous statements, one can show that $q_{\ell}(F\Diamond G)\leq q_{\ell}(F)q_{\ell}(G)$, for all $F,G\in\mathcal{A}^{**}$ and $\ell\in\mathbb{N}$.
\end{proof}

\begin{theorem}\label{aaaccclllfrechet}
If $(\mathcal{A},p_{\ell})$ if a Fr\'echet algebra, then,
$(\mathcal{A}^{**},q_{\ell})$ is also a Fr\'echet algebra.
\end{theorem}

\begin{proof}
Similar to the proof of \cite[Proposition 25.9]{Meise} and by applying Lemma \ref{plllqlllfunde}, one can show that $\mathcal{A}^{**}$ is complete and consequently is a Fr\'echet algebra. 
\end{proof}

\section{Arens Regularity of Fr\'echet algebras}\label{section3}

Let $(\mathcal{A},p_{\ell})$ be a Fr\'echet algebra. By Lemma \ref{plllqlllfunde} (ii), for every $f\in\mathcal{A}^*$ there exists a smallest $\ell\in\mathbb{N}$ such that $q_{M_{\ell}}(f)<\infty$. Thus, by setting $U_f:=M_{\ell}$, we have $f|_{U_f}\in\ell^{\infty}(U_f)$.
Note that  $U_f$ is an idempotent neighborhood of $0$ and is a semigroup. For this reason and due to Eberlein \cite{Eberlein}, we can define weakly almost periodic functions on Fr\'echet algebras in such a way that Grothendieck theorem holds. 

\begin{definition}\label{D1}
Let $\mathcal{A}$ be a Fr\'echet algebra and $f\in\mathcal{A}^*$. We say that $f\in wap(\mathcal{A})$ if and only if $\{(a\cdot f)|_{U_f}: a\in {U_f}\overline{\}}^w$ is weakly compact in $\ell^{\infty}(U_f)$.
\end{definition}

By applying Young  theorem, mentioned at the beginning of section \ref{S0S}, the following proposition is immediate.

\begin{proposition}\label{Grothendieckrr}
Let $\mathcal{A}$ be a Fr\'echet algebra. Then $f\in wap(\mathcal{A})$ if and only if for  sequences $(a_m)$ and $(b_n)$ in $U_f$ with distinct elements
$$\lim_m\lim_nf(a_mb_n)=\lim_n\lim_mf(a_mb_n),$$
whenever both repeated limits exist. 
\end{proposition}

\begin{lemma}\label{rwrwrw1111122}
Let $(\mathcal{A},p_{\ell})$ be a Fr\'echet algebra and $f\in wap(\mathcal{A})$. Then $a\cdot f$ and $f\cdot a$ belong to $wap(\mathcal{A})$, for every $a\in\mathcal{A}$.
\end{lemma}

\begin{proof}
By Proposition \ref{Grothendieckrr}, for sequences $(a_m)$ and $(b_n)$ in $U_f$ with distinct elements, we have 
$$\lim_m\lim_nf(a_mb_n)=\lim_n\lim_mf(a_mb_n),$$
whenever both limits exist. Recall that $U_f=M_{\ell}$, for some $\ell\in\mathbb{N}$.
Since $\frac{aa_m}{p_{\ell}(a)+1}\in U_f$, for every $a\in\mathcal{A}$ and $m\in\mathbb{N}$, we have
$$
\lim_m\lim_nf(\frac{aa_m}{p_{\ell}(a)+1}b_n)=\lim_n\lim_mf(\frac{aa_m}{p_{\ell}(a)+1}b_n),
$$
and consequently
$$
\lim_m\lim_n(f\cdot a)(a_mb_n)=\lim_n\lim_m(f\cdot a)(a_mb_n).
$$
Similarly, one can show that
$$
\lim_m\lim_n(a\cdot f)(a_mb_n)=\lim_n\lim_m(a\cdot f)(a_mb_n)\;\;\;\;\;\;(a\in\mathcal{A}).
$$
Therefore, $f\cdot a$ and $a\cdot f\in wap(\mathcal{A})$, for every $a\in\mathcal{A}$.  
\end{proof}

\begin{proposition}\label{WAPwapa}
For a Fr\'echet algebra $(\mathcal{A},p_{\ell})$, we have $WAP(\mathcal{A})= wap(\mathcal{A})$.
\end{proposition}

\begin{proof}
If $f\in WAP(\mathcal{A})$, then there exists a zero neighborhood $U$ such that the set $\{a\cdot f: a\in U\overline{\}}^w$ is weakly compact in $\mathcal{A}^*$. In addition, $f|_{U_f}\in\ell^{\infty}(U_f)$ and there are $\ell_0\in\mathbb{N}$ and $0<\delta<1$ so that the idempotent set $\{a\in\mathcal{A}: p_{\ell_0}(a)<\delta\}$ is a subset of $U\cap U_f$. By choosing the smallest $\ell_0$ with this properties, set 
$$V_f:=\{a\in\mathcal{A}: p_{\ell_0}(a)<\delta\}.$$
 Thus, $f|_{V_f}\in\ell^{\infty}(V_f)$. Now, we show that $\{(a\cdot f)|_{V_f}: a\in V_f\overline{\}}^w$ is weakly compact in $\ell^{\infty}(V_f)$ and so $f\in wap(\mathcal{A})$. Define
$$
T:\{a\cdot f: a\in V_f\overline{\}}^w\subseteq\mathcal{A}^*\rightarrow \ell^{\infty}(V_f),\;\;\;g\mapsto g|_{V_f}.
$$
Obviously, $T$ is linear. Let $G\in\ell^{\infty}(V_f)^*$. Define $F:\{a\cdot f: a\in V_f\overline{\}}^w\rightarrow\mathbb{C}$ by $F(g):=G(g|_{V_f})$,
for each $g\in\{a\cdot f: a\in V_f\overline{\}}^w$. Clearly, $F$ is linear and continuous. Indeed, if the net $(g_{\alpha})$ is convergent to some $g$ in $\{a\cdot f: a\in V_f\overline{\}}^w$ with respect to the relative topology of the strong topology on $\mathcal{A}^*$, then
$$
\|g_{\alpha}|_{V_f}-g|_{V_f}\|_{\infty}=\sup_{a\in V_f}|(g_{\alpha}-g)(a)|\longrightarrow_{\alpha} 0.
$$
By continuity of $G$, we have
$G(g_{\alpha}|_{V_f})\longrightarrow_{\alpha} G(g|_{V_f})$ and so $F$ is continuous.
If $g_{\alpha}\stackrel{w}{\longrightarrow}_{\alpha}g$ in $\{a\cdot f: a\in V_f\overline{\}}^w$, then
$$
G(g_{\alpha}|_{V_f})=F(g_{\alpha})\longrightarrow_{\alpha} F(g)=G(g|_{V_f}),
$$
and $g_{\alpha}|_{V_f}\stackrel{w}{\longrightarrow_{\alpha}}g|_{V_f}$ in $\ell^{\infty}(V_f)$. Thus, $T$ is weakly continuous. We claim that the range of $T$ is $\{a\cdot f|_{V_f}: a\in V_f\overline{\}}^w$. 
For $g\in\{a\cdot f: a\in V_f\overline{\}}^w$, there exists a net $(a_{\alpha})$ in $V_f$ such that $a_{\alpha}\cdot f\stackrel{w}{\longrightarrow}_{\alpha}g$. By continuity of $T$, we have $(a_{\alpha}\cdot f)|_{V_f}\stackrel{w}{\longrightarrow}_{\alpha}g|_{V_f}$ and $g|_{V_f}\in\{(a\cdot f)|_{V_f}:a\in V_f\overline{\}}^w$. Thus, 
$$T\big{(}\{a\cdot f: a\in V_f\overline{\}}^w\big{)}\subseteq\{(a\cdot f)|_{V_f}: a\in V_f\overline{\}}^w.$$ 
Conversely, if $h\in\{(a\cdot f)|_{V_f}: a\in V_f\overline{\}}^w$, then there exists a net $(a_{\alpha})\subseteq V_f$ so that $T(a_{\alpha}\cdot f)\stackrel{w}{\longrightarrow}_{\alpha}h$. Hence, by $(a_{\alpha}\cdot f)\subseteq\{a\cdot f: a\in V_f\overline{\}}^w$, we have $h\in \overline{T(\{a\cdot f: a\in V_f\overline{\}}^w)}^w$ and
$$
\{(a\cdot f)|_{V_f}: a\in V_f\overline{\}}^w\subseteq T(\{a\cdot f: a\in V_f\overline{\}}^w).
$$
Thus, $WAP(\mathcal{A})\subseteq wap(\mathcal{A})$, since the continuous image of a compact set is compact. 

To prove inclusion $wap(\mathcal{A})\subseteq WAP(\mathcal{A})$, let $f\in wap(\mathcal{A})$. Therefore, the set $\{(a\cdot f)|_{U_f}: a\in U_f\overline{\}}^w$ is weakly compact in $\ell^{\infty}(U_f)$, where $U_f=M_{\ell}$ for some $\ell\in\mathbb{N}$. For convenience, let $B:=\{a\cdot f: a\in U_f\overline{\}}^w$ and consider the linear map
$T:B\rightarrow \ell^{\infty}(U_f)$, $g\mapsto g|_{U_f}$.
For every $g,h\in B$ with $g|_{U_f}=h|_{U_f}$ we have  
$g(\frac{a}{p_{\ell}(a)+1})=h(\frac{a}{p_{\ell}(a)+1})$ ($a\in\mathcal{A}$) and so $g=h$. Thus, $T^*:\ell^{\infty}(U_f)^*\rightarrow B^*$ ($F\mapsto F\circ T$) is onto. Hence, for every $G\in\mathcal{A}^{**}$ there exists $F\in \ell^{\infty}(U_f)^*$ such that $G|_{B}=F\circ T$. Let $(a_{\alpha})\subseteq  U_f$. By assumption, there exists a subnet of $((a_{\alpha}\cdot f)|_{U_f})$ which is weakly convergent to some $g$ in $\{(a\cdot f)|_{U_f}: a\in U_f\overline{\}}^w$. Without loss of generality we can suppose that $(a_{\alpha}\cdot f)|_{U_f}\stackrel{w}{\longrightarrow}_{\alpha}g$. Thus,
$G(a_{\alpha}\cdot f)\longrightarrow_{\alpha} F(g)$ and 
the net $(a_{\alpha}\cdot f)$ is weakly convergent in $B$. For this reason, $B$ is weakly compact in $\mathcal{A}^*$.
\end{proof}

Let $X$ and $Y$ be non-empty, locally compact spaces, and $f:X\times Y\rightarrow \mathbb{C}$ be a bounded, separately continuous function. Therefore, $\{y\cdot f: y\in Y\overline{\}}^w$ is weakly compact in $C(\beta X)$, where $\beta X$ is the Stone-\u Cech compactification of $X$, if and only if 
$\lim_m\lim_n f(x_m,y_n)=\lim_n\lim_mf(x_m,y_n)$,
whenever both limits exist and $(x_m)$ and $(y_n)$ are sequences of distinct points
in $X$ and $Y$, respectively; see \cite[Theorem 3.3]{Dales}. Now, the following theorem can be obtained.

\begin{theorem}\label{T1}
Let $\mathcal{A}$ be a Fr\'echet algebra and $f\in\mathcal{A}^*$. The following statements are equivalent: 
\begin{enumerate}
\item[(i)]
$f\in WAP(\mathcal{A})$;
\item[(ii)]
$f\in wap(\mathcal{A})$;
\item[(iii)] 
$\{a\cdot(f|_{U_f}): a\in U_f\overline{\}}^w$ is weakly compact in $\ell^{\infty}(U_f)$;
\item[(iv)] 
$\big{\{}a\cdot (f|_{U_f})^{\beta}: a\in U_f\overline{\big{\}}}^w$ is weakly compact in $C(\beta (U_f))$;
\item[(v)] 
for all sequences $(a_m)$ and $(b_n)$ in $U_f$ with distinct elements
$$\lim_m\lim_nf(a_mb_n)=\lim_n\lim_mf(a_mb_n),$$
whenever both repeated limits exist. 
\end{enumerate}
\end{theorem}

\begin{proof}
(i)$\Leftrightarrow$(ii). 
By applying Proposition \ref{WAPwapa}, it is immediate.

(ii)$\Leftrightarrow$(iii). 
It is obvious, by Definition of $wap(\mathcal{A})$.

(iii)$\Rightarrow$(iv).
Consider the map
$T:(\ell^{\infty}(U_f),w)\rightarrow (C(\beta U_f),w)$ ($g\mapsto g^{\beta}$), 
where $g^{\beta}$ is defined by 
$$g^{\beta}(x)=\lim g(x_{\alpha})\;\;\;\;\;\;(x\in \beta U_f)$$
for some 
$(x_{\alpha})\subseteq U_f$ with $\hat{x}_{\alpha}\stackrel{w^*}{\longrightarrow}x$.
In a similar manner to the proof of Proposition \ref{WAPwapa}, one can show that $T$ is continuous and 
$$T\big{(}\{a\cdot g: a\in U_f\overline{\}}^w\big{)}=\{a\cdot (g^{\beta}): a\in U_f\overline{\}}^w,$$
for every $g\in\ell^{\infty}(U_f)$. This completes the proof.

(iv)$\Leftrightarrow$(v). In \cite[Theorem 3.3]{Dales}, it is enough to set $X=Y=U_f$ under the discrete topology.

(v)$\Leftrightarrow$(ii). It is Young theorem, pointed out in Proposition \ref{Grothendieckrr}.
\end{proof}

\begin{definition}\label{7120final}
A Fr\'echet algebra $\mathcal{A}$ is called Arens regular if $WAP(\mathcal{A})=\mathcal{A}^*$.
\end{definition}

\begin{proposition}\label{lll6262}
For a Fr\'echet algebra $(\mathcal{A},p_{\ell})$, the following are equivalent:
\begin{enumerate}
\item[(i)] $\mathcal{A}$ is Arens regular;
\item[(ii)] $\square=\Diamond$.
\end{enumerate}
\end{proposition}

\begin{proof}
Let $\mathcal{A}$ be Arens regular, $f\in\mathcal{A}^*$, and $F,G\in \mathcal{A}^{**}$.
By using \cite[Theorem 22.13]{Meise} and \cite[Proposition 3.1]{Dales}, there exist bounded sequences $(a_m)$ and $(b_n)$ in $\mathcal{A}$ such that 
$$
F\square G(f)=\lim_{m}\lim_{n}f(a_m b_n)\;\;\;\text{and}\;\;\; F\Diamond G(f)=\lim_{n}\lim_{m}f(a_m b_n).$$
Moreover, by assumption, the set $\{(a\cdot f)|_{U_f}: a\in U_f\overline{\}}^w$ is weakly compact in $\ell^{\infty}(U_f)$, whenever $U_f=M_{\ell}$ for some $\ell\in\mathbb{N}$.
By using \cite[Remark 23.2]{Meise}, we have $K_1:=\sup_{m\in\mathbb{N}}p_{\ell}(a_m)<\infty$ and $K_2:=\sup_{n\in\mathbb{N}}p_{\ell}(b_n)<\infty$. Since the sequences $(\frac{a_m}{K_1+1})$ and $(\frac{b_n}{K_2+1})$ belong to $U_f$, by Theorem \ref{T1}, we have
$$
\lim_m\lim_nf(\frac{a_m}{K_1+1}\cdot\frac{b_n}{K_2+1})=\lim_n\lim_mf(\frac{a_m}{K_1+1}\cdot\frac{b_n}{K_2+1}),
$$
whenever both limits exist. Thus, $F\square G(f)=F\Diamond G(f)$. Similarly, one can show that the converse holds. 
\end{proof}

Due to Theorem \ref{T1} and Proposition \ref{lll6262}, we can conclude the following proposition which its proof is analogous to the Banach case and so will be omitted.

\begin{proposition}
For a Fr\'echet algebra $\mathcal{A}$, the following statements are equivalent:
\begin{enumerate}
\item[(i)] $\mathcal{A}$ is Arens regular;
\item[(ii)] for every $G\in\mathcal{A}^{**}$, the linear map $F\mapsto G\square F$ is continuous on $\mathcal{A}^{**}$ with respect to the weak$^*$ topology;
\item[(iii)] for every $G\in\mathcal{A}^{**}$, the linear map $F\mapsto F\Diamond G$ is continuous on $\mathcal{A}^{**}$ with respect to the weak$^*$ topology.
\end{enumerate}
\end{proposition}

Note that one can define $WAP(\mathcal{A})$ for every locally convex algebra $\mathcal{A}$. Thus similar to the Definition \ref{7120final}, Arens regular locally convex algebras are defined. Consequently, the following result is immediate. 

\begin{proposition}
Let $\mathcal{A}$ be a Fr\'echet algebra. The following statements are equivalent:
\begin{enumerate}
\item[(i)]
$\mathcal{A}^{**}$ is Arens regular;
\item[(ii)]
$\mathcal{A}$ and $WAP(\mathcal{A})^*$ are Arens regular;
\item[(iii)]
$\mathcal{A}$ is Arens regular and $WAP(\mathcal{A})^{**}=WAP(\mathcal{A}^{**})$.
\end{enumerate}
\end{proposition}

\begin{proof}
If $\mathcal{A}^{**}$ is Arens regular, then $\mathcal{A}$ is Arens regular and $WAP(\mathcal{A})=\mathcal{A}^*$. For this reason, we have
$$
WAP(WAP(\mathcal{A})^*)=WAP(\mathcal{A}^{**})=\mathcal{A}^{***}=WAP(\mathcal{A})^{**}.
$$
Similarly, the converse holds.
\end{proof}

\begin{remark}
Pym gives an example of an Arens regular Banach algebra $\mathcal{A}$ such that $\mathcal{A}^{**}$ is not Arens regular \cite{Pym1}. 
\end{remark}

In \cite{Rejali}, Rejali and Vishki showed that for a locally compact (topological) group $G$ with a weight function $w$ on it, the following statements are equivalent:
\begin{enumerate}
\item[(i)]
$L^1(G,w)^{**}$ is Arens regular;
\item[(ii)]
$L^1(G,w)$ is Arens regular;
\item[(iii)]
$L^1(G)$ is Arens regular or $\Omega:G\times G\rightarrow (0,1]$, defined by
$$
\Omega(x,y)=\frac{w(xy)}{w(x)w(y)}\;\;\;\;\;(x,y\in G),
$$
is $0$-cluster, that is 
$$\lim_m\lim_n\Omega(x_m,y_n)=0=\lim_n\lim_m\Omega(x_m,y_n),$$
for sequences $(x_m)$ and $(y_n)$ in $G$ with distinct elements.
\end{enumerate}
Moreover, Baker and Rejali showed that $\ell^1(S,w)$ is Arens regular, whenever $\ell^1(S)$ is  \cite{Baker}. 
Now, the following question is natural.

\begin{question}
Let $S$ be a semigroup and $w$ be a weight function on $S$. Are the following statements equivalent?
\begin{enumerate}
\item[(i)]
$\ell^1(S,w)^{**}$ is Arens regular.
\item[(ii)]
$\ell^1(S,w)$ is Arens regular.
\item[(iii)]
$\ell^1(S)$ is Arens regular or $\Omega$ is $0$-cluster.
\end{enumerate}
\end{question}

Later, Rejali et al. showed that if the set of all idempotents of an inverse semigroup $S$ is finite, then $\ell^1(S)$ is Arens regular if and only if $S$ is finite; see \cite[Theorem 3.8]{Abtahi}. Also in this case, if $S$ admits a bounded below weight for which $\Omega$ is $0$-cluster, then $S$ is countable; see \cite[Theorem 3.4]{Khodsiani}. By arguments above, another question arises.

\begin{question}
Let $S$ be an inverse semigroup with finitely many idempotents and $w$ be a weight function on $S$. Are the following statements equivalent?
\begin{enumerate}
\item[(i)]
$\ell^1(S,w)^{**}$ is Arens regular.
\item[(ii)]
$\ell^1(S,w)$ is Arens regular.
\item[(iii)]
$S$ is finite or $\Omega$ is $0$-cluster.
\end{enumerate}
\end{question}

\begin{remark}
In \cite[Example 3.9]{Abtahi}, the authors showed that there exists an inverse semigroup $S$ with infinitely many idempotents such that $\ell^1(S)$ is  Arens regular. 
\end{remark}

\section{\bf Product of Arens regualr Fr\'echet algebras}\label{section4}

Throughout this section, $(\mathcal{A}_n,p^n_{\ell})$ is a countable sequence of Fr\'echet algebras. We define
\begin{eqnarray*}
&&\ell^1\text{-}\prod_{n\in\mathbb{N}}\mathcal{A}_n=\big{\{}(a_n)\in\prod_{n\in\mathbb{N}} \mathcal{A}_n: \sum_{\ell,n=1}^{\infty}p^n_{\ell}(a_n)<\infty\big{\}}\;\;\;\;\;\text{and} \\
&&\ell^{\infty}\text{-}\prod_{n\in\mathbb{N}}\mathcal{A}_n=\big{\{}(a_n)\in\prod_{n\in\mathbb{N}} \mathcal{A}_n:\sup_{\ell,n\in\mathbb{N}}p^n_{\ell}(a_n)<\infty\big{\}}.
\end{eqnarray*}
Moreover, $c_0\text{-}\prod_{n\in\mathbb{N}}\mathcal{A}_n$ denotes the closed subalgebra of $\ell^{\infty}\text{-}\prod_{n\in\mathbb{N}}\mathcal{A}_n$ consisting of all elements vanishing at infinity and $c_{00}\text{-}\prod_{n\in\mathbb{N}}\mathcal{A}_n$ denotes the closed subalgebra of $\ell^1\text{-}\prod_{n\in\mathbb{N}}\mathcal{A}_n$  consisting of all elements $(a_n)$ such that $a_n\not=0$ only for finitely many $n\in\mathbb{N}$. In the following theorem, we show that $\ell^1\text{-}\prod_{n\in\mathbb{N}}\mathcal{A}_n$ and $\ell^{\infty}\text{-}\prod_{n\in\mathbb{N}}\mathcal{A}_n$ are Fr\'echet algebras. Thus, by \cite[Proposition 25.3]{Meise}, $c_0\text{-}\prod_{n\in\mathbb{N}}\mathcal{A}_n$ and $c_{00}\text{-}\prod_{n\in\mathbb{N}}\mathcal{A}_n$ are Fr\'echet algebras.

\begin{theorem}\label{T1S3T11}
The following statements hold.
\begin{enumerate}
\item[(i)]
$\ell^1\text{-}\prod_{n\in\mathbb{N}}\mathcal{A}_n$ is a Fr\'echet algebra under the pointwise product and the fundamental system of seminorms $(r_{\ell})_{\ell\in\mathbb{N}}$, defined by 
$$r_{\ell}((a_n))=\sum_{n=1}^{\infty}p^n_{\ell}(a_n)\;\;\;\;\;\;\;\;\big{(}\ell\in\mathbb{N},\;(a_n)\in \ell^1\text{-}\prod_{n\in\mathbb{N}}\mathcal{A}_n\big{)}.$$
\item[(ii)]
$\ell^{\infty}\text{-}\prod_{n\in\mathbb{N}}\mathcal{A}_n$ is a Fr\'echet algebra under the pointwise product and the fundamental system of seminorms $(\gamma_{\ell})_{\ell\in\mathbb{N}}$, defined by 
$$\gamma_{\ell}((a_n))=\sup_{n\in\mathbb{N}}p^n_{\ell}(a_n)\;\;\;\;\;\;\;\;\big{(}\ell\in\mathbb{N},\;(a_n)\in \ell^{\infty}\text{-}\prod_{n\in\mathbb{N}}\mathcal{A}_n\big{)}.$$
\end{enumerate}
\end{theorem}

\begin{proof}
We only prove (i). 
Easily, one can show that the seminorms $(r_{\ell})_{\ell\in\mathbb{N}}$ are submultiplicative
and have the properties of \cite[Lemma 22.4]{Meise}. 
It is enough to show that $\ell^1\text{-}\prod_{n\in\mathbb{N}}\mathcal{A}_n$ is complete. Let $(a_m)$ be a Cauchy sequence in $\ell^1\text{-}\prod_{n\in\mathbb{N}}\mathcal{A}_n$ and $a_m=(x^m_k)_k$, for each $m\in\mathbb{N}$. Thus, given $\varepsilon>0$ and $\ell\in\mathbb{N}$, there exists $N_{\ell}>0$ such that 
$r_{\ell}(a_m-a_n)<\varepsilon$ ($m,n\geq N_{\ell}$).
For this reason, we have
\begin{equation}\label{infty mmm}
\sum_{k=1}^{\infty}p_{\ell}^k(x^m_k-x^n_k)<\varepsilon,
\end{equation}
and $p_{\ell}^k(x^m_k-x^n_k)<\varepsilon$, for $m,n\geq N_{\ell}$. Therefore, $(x^m_k)_k$ is a Cauchy sequence in $\mathcal{A}_k$ and there exists $x_k\in\mathcal{A}_k$ such that $x^m_k\longrightarrow_{m} x_k$. Let $a:=(x_k)$. By using (\ref{infty mmm}), we have $\sum_{k=1}^{\infty}p_{\ell}^k(x^m_k-x_k)<\varepsilon$ and so $r_{\ell}(a_m-a)\longrightarrow 0$. Also, for every $\varepsilon>0$ and $\ell\in\mathbb{N}$ there exists $N_{\ell}>0$ such that
$\sum_{k=1}^{\infty}p_{\ell}^k(x^n_k-x_k)<\varepsilon$ ($n\geq N_{\ell}$). Thus,
$$
\sum_{k=1}^{\infty}p_{\ell}^k(x_k)\leq \sum_{k=1}^{\infty}p_{\ell}^k(x^{N_{\ell}}_k-x_k)+\sum_{k=1}^{\infty}p_{\ell}^k(x^{N_{\ell}}_k)<\varepsilon+r_{\ell}(x^{N_{\ell}}_k)<\infty,
$$
and $a\in\ell^1\text{-}\prod_{n\in\mathbb{N}}\mathcal{A}_n$.
\end{proof}

Now, consider
\begin{eqnarray*}
&&\ell^1\text{-}\prod_{n\in\mathbb{N}}\mathcal{A}^{**}_n=\big{\{}(F_n)\in\prod_{n\in\mathbb{N}}\mathcal{A}^{**}_n: \sum_{\ell,n=1}^{\infty}q^n_{\ell}(F_n)<\infty\big{\}}\;\;\;\;\;\text{and} \\
&&\ell^{\infty}\text{-}\prod_{n\in\mathbb{N}}\mathcal{A}^{**}_n=\big{\{}(F_n)\in\prod_{n\in\mathbb{N}}\mathcal{A}^{**}_n: \sup_{\ell,n\in\mathbb{N}} q^n_{\ell}(F_n)<\infty\big{\}},
\end{eqnarray*}
where 
$$
q^n_{\ell}(F_n)=q_{M^{\circ}_{(n,\ell)}}(F_n)=\sup \big{\{}|F_n(f)|:q_{M_{(n,\ell)}}(f)\leq 1,\;\text{for every $f\in\mathcal{A}^*_n$}\big{\}},
$$
for each $n,\ell\in\mathbb{N}$. By using Theorem \ref{aaaccclllfrechet}, the following result is obvious. 

\begin{corollary}
The submultiplicative seminorms 
$\sum_{n=1}^{\infty}q^n_{\ell}(\cdot)$ and $\sup_{n\in\mathbb{N}} q^n_{\ell}(\cdot)$ $(\ell\in\mathbb{N})$
introduce a fundamental system on 
$\ell^1\text{-}\prod_{n\in\mathbb{N}}\mathcal{A}^{**}_n$ and $\ell^{\infty}\text{-}\prod_{n\in\mathbb{N}}\mathcal{A}^{**}_n$, respectively.
Then $\ell^1\text{-}\prod_{n\in\mathbb{N}}\mathcal{A}^{**}_n$ and $\ell^{\infty}\text{-}\prod_{n\in\mathbb{N}}\mathcal{A}^{**}_n$ are Fr\'echet algebras.
\end{corollary}

Consider
{\small$$
\ell^1\text{-}\prod_{n\in\mathbb{N}}\mathcal{A}^*_n=\big{\{}(f_n)\in\prod_{n\in\mathbb{N}}\mathcal{A}^*_n: \sum_{n=1}^{\infty}q_{M_{(n,\ell_n)}}(f_n)<\infty,\;\text{for some sequence $(\ell_n)$ in $\mathbb{N}$}\big{\}}
$$}
and
{\small$$
\ell^{\infty}\text{-}\prod_{n\in\mathbb{N}}\mathcal{A}^*_n=\big{\{}(f_n)\in\prod_{n\in\mathbb{N}}\mathcal{A}^*_n:\sup_{n\in\mathbb{N}}q_{M_{(n,\ell_n)}}(f_n)<\infty,\;\text{for some sequence $(\ell_n)$ in $\mathbb{N}$}\big{\}},
$$}
where by Lemma \ref{plllqlllfunde} (ii), $q_{M_{(n,\ell_n)}}(f_n)=\sup_{x\in M_{(n,\ell_n)}}|f_n(x)|<\infty$.
Now, we can conclude the following proposition.

\begin{proposition}\label{BBperp}
For $\mathcal{A}=\ell^1\text{-}\prod_{n\in\mathbb{N}}\mathcal{A}_n$, the following statements hold.
\begin{enumerate}
\item[(i)]
$\mathcal{A}^*\cong\ell^{\infty}\text{-}\prod_{n\in\mathbb{N}}\mathcal{A}^*_n$.
\item[(ii)]
$\mathcal{A}^{**}\cong B^*\oplus B^{\perp}$, where
$$
B=\big{\{}(f_n)\in\prod_{n\in\mathbb{N}} \mathcal{A}^*_n:q_{M_{(n,\ell_n)}}(f_n)\longrightarrow 0,\;\text{for some sequence $(\ell_n)$ in $\mathbb{N}$}\big{\}}
$$
and 
$$
B^{\perp}=\big{\{}F\in\mathcal{A}^{**}:F(B)=\{0\}\big{\}}.
$$
\end{enumerate}
\end{proposition}

\begin{proof}
(i). Define
$\phi:\ell^{\infty}\text{-}\prod_{n\in\mathbb{N}}\mathcal{A}^*_n\rightarrow \mathcal{A}^*$ by $\phi((f_n))((a_n)):=\sum_{n=1}^{\infty}f_n(a_n)$,
for all $(f_n)\in\ell^{\infty}\text{-}\prod_{n\in\mathbb{N}}\mathcal{A}^*_n$ and $(a_n)\in\mathcal{A}$. Suppose that $(f_n)\in\ell^{\infty}\text{-}\prod_{n\in\mathbb{N}}\mathcal{A}^*_n$. Then, $K:=\sup_{n\in\mathbb{N}} q_{M_{(n,\ell_n)}}(f_n)<\infty$ for some sequence $(\ell_n)$. By applying Lemma \ref{plllqlllfunde} (iii), we have
$$
\sum_{n=1}^{\infty}|f_n(a_n)|\leq \sum_{n=1}^{\infty}q_{M_{(n,\ell_n)}}(f_n)p^n_{\ell_n}(a_n)\leq K\sum_{n=1}^{\infty}p^n_{\ell_n}(a_n)<\infty.
$$
For $f\in\mathcal{A}^*$, let $f_n:=f\circ j_n$, where
$j_n:\mathcal{A}_n\rightarrow\mathcal{A}$ ($n\in\mathbb{N}$) is the canonical map.
Clearly, $f$ and $j_n$ are linear and continuous and consequently $f_n\in \mathcal{A}^*_n$ ($n\in\mathbb{N}$). For this reason, we have
$$
\sum_{n=1}^{\infty}f_n(a_n)=f(\sum_{n=1}^{\infty}j_n(a_n))=f((a_n))\;\;\;\;\;\;((a_n)\in\mathcal{A})
$$
and so $\phi$ is onto. Also, if $\phi((f_n))=\phi((g_n))$ for $(f_n), (g_n)\in \ell^{\infty}\text{-}\prod_{n\in\mathbb{N}}\mathcal{A}^*_n$, then 
$$
\sum_{n=1}^{\infty}f_n(j_m(a_m))=\sum_{n=1}^{\infty}g_n(j_m(a_m))\;\;\;\;\;\;((a_n)\in\mathcal{A}),
$$
and $(f_n)=(g_n)$. To prove $\mathcal{A}^*\cong\ell^{\infty}\text{-}\prod_{n\in\mathbb{N}}\mathcal{A}^*_n$,  let $(f^{\alpha})$ be convergent to some $f$ in $\ell^{\infty}\text{-}\prod_{n\in\mathbb{N}}\mathcal{A}^*_n$, where $f=(f_n)_n$ and
$f^{\alpha}=(f^{\alpha}_n)_n$, for each $\alpha$. 
Thus, we have $\sup_{n\in\mathbb{N}}q_{M_{(n,\ell_n)}}(f^{\alpha}_n-f_n)\rightarrow 0$ and  $q_{M_{(n,\ell_n)}}(f^{\alpha}_n-f_n)\rightarrow 0$, for some sequence $(\ell_n)$ in $\mathbb{N}$.
Consider a bounded subset $M$ of $\mathcal{A}$. There exists $\lambda>0$ such that
$M\subseteq \lambda \prod_{n\in\mathbb{N}}M_{(n,\ell_n)}$. Therefore, $q_M(f^{\alpha}-f)\rightarrow 0$ and $f^{\alpha}\rightarrow f$ in $\mathcal{A}^*$.
Conversely, let $(f^{\alpha})$ be convergent to some $f$ in $\mathcal{A}^*$, where $f=(f_n)_n$ and $f^{\alpha}=(f^{\alpha}_n)_n$, for each $\alpha$. We have
$$
q_{M_{(n,\ell_n)}}(f^{\alpha}_n-f_n)=\sup_{x\in M_{(n,\ell_n)}}|f^{\alpha}_n(x)-f_n(x)|= 
\sup_{a\in j_n(M_{(n,\ell_n)})}|f^{\alpha}(a)-f(a)|, 
$$
for some sequence $(\ell_n)$ in $\mathbb{N}$. Since $(f^{\alpha}-f)$ is bounded on $j_n(M_{(n,\ell_n)})$ ($n\in\mathbb{N}$), by assumption, $f^{\alpha}\rightarrow f$ in $\ell^{\infty}\text{-}\prod_{n\in\mathbb{N}}\mathcal{A}^*_n$.

(ii).
We claim that $B^*\cong \ell^1\text{-}\prod_{n\in\mathbb{N}}\mathcal{A}^{**}_n$. Define $\phi:\ell^1\text{-}\prod_{n\in\mathbb{N}}\mathcal{A}^{**}_n\rightarrow B^*$ by $\phi((F_n))((f_n))=\sum_{n=1}^{\infty}F_n(f_n)$, for every $(F_n)\in\ell^1\text{-}\prod_{n\in\mathbb{N}}\mathcal{A}^{**}_n$ and $(f_n)\in B$. Since $q_{M_{(n,\ell_n)}}(f_n)\rightarrow 0$ and $\sum_{n=1}^{\infty}q^n_{\ell_n}(F_n)<\infty$, for some sequence $(\ell_n)$, we have
$$
\sum_{n=1}^{\infty}|F_n(f_n)|\leq \sum_{n=1}^{\infty}q^n_{\ell_n}(F_n)q_{M_{(n,\ell_n)}}(f_n)<\infty,
$$
and $\phi$ is well-defined. Similar to (i), one can show that $\phi$ is a bijective map. 
Let $(F^{\alpha})$ be convergent to some $F$ in $\ell^1\text{-}\prod_{n\in\mathbb{N}}\mathcal{A}^{**}_n$, where $F=(F_n)_n$ and
$F^{\alpha}=(F^{\alpha}_n)_n$, for each $\alpha$. Then 
$\sum_{n=1}^{\infty}q^n_{\ell}(F^{\alpha}_n-F_n)\rightarrow_{\alpha} 0$ and so 
\begin{eqnarray}\label{BCBCBC1362}
\;\;\;\;\;\;\;\;\;\;\;\;\;\;\;\;\;\;\;\;\;\;\;\;\;\;\;\;\;\;\;\;\;\;\;\;\;\;\;
q^n_{\ell}(F^{\alpha}_n-F_n)=q_{M^{\circ}_{(n,\ell)}}(F^{\alpha}_n-F_n)\rightarrow_{\alpha} 0\;\;\;\;\;\;\;(n,\ell\in\mathbb{N}).
\end{eqnarray}
If $M\subseteq B$ is a bounded set, then there are $(\lambda_n)\subseteq\mathbb{R}^{+}$ and $(\ell_n)\subseteq\mathbb{N}$ such that $M\subseteq \prod_{n\in\mathbb{N}}\lambda_n M^{\circ}_{(n,\ell_n)}$. Thus, by using (\ref{BCBCBC1362}), $q_M(F^{\alpha}-F)\rightarrow_{\alpha} 0$. Due to the open mapping theorem \cite[Theorem 24.30]{Meise}, the converse holds. By similar arguments, one can see that $\mathcal{A}^{**}\cong B^*\oplus B^{\perp}$.
\end{proof}

We now state the main result of this paper which is a generalization of \cite[Theorem 6]{Arikan} for Fr\'echet algebras. 

\begin{theorem}
The Fr\'echet algebra $\ell^1\text{-}\prod_{n\in\mathbb{N}}\mathcal{A}_n$ is Arens regular if and only if each $\mathcal{A}_n$ is Arens regular.
\end{theorem}

\begin{proof}
Let $\mathcal{A}=\ell^1\text{-}\prod_{n\in\mathbb{N}}\mathcal{A}_n$ and each $\mathcal{A}_n$ be Arens regular. Consider $a=(a_n)$ in $\mathcal{A}$ and $f=(f_n)$ in $\mathcal{A}^*$. We show that $a\cdot f\in B$, where $B$ is defined as in Proposition \ref{BBperp}. By Lemma \ref{plllqlllfunde} (ii) and (iii), for every $x\in M_{(n,\ell_n)}$ we have
$$
|a_n\cdot f_n(x)|\leq q_{M_{(n,\ell_n)}}(f_n)p^n_{\ell_n}(x)p^n_{\ell_n}(a_n)\;\;\;\;\;\;(n\in\mathbb{N}),
$$
for some sequence $(\ell_n)$ in $\mathbb{N}$. By assumption, $p^n_{\ell_n}(x)\leq 1$ and $p^n_{\ell_n}(a_n)\rightarrow_{n}0$ and so $|a_n\cdot f_n(x)|\rightarrow_{n}0$, for every $x\in M_{(n,\ell_n)}$. Thus, 
$$q_{M_{(n,\ell_n)}}(a_n\cdot f_n)=\sup_{x\in M_{(n,\ell_n)}}|a_n\cdot f_n(x)|\rightarrow_{n}0,$$
and $a\cdot f\in B$. For this reason, $\mathcal{A}\cdot\mathcal{A}^*\subseteq B$ and similarly, $\mathcal{A}^*\cdot \mathcal{A}\subseteq B$. Therefore,
\begin{equation}\label{F11}
\mathcal{A}\cdot\mathcal{A}^*\cup \mathcal{A}^*\cdot \mathcal{A}\subseteq B.
\end{equation}
Let $F=(F_1,F_2)\in\mathcal{A}^{**}$, for some $F_1\in B^*$ and $F_2\in B^{\perp}$. According to the previous arguments, we have
$f\cdot F_2(a)=F_2(a\cdot f)=0$
and so $f\cdot F(a)=f\cdot F_1(a)$, for every $a\in\mathcal{A}$ and $f\in\mathcal{A}^*$.
Now, for $f\in\mathcal{A}^*$ we have $f\cdot F_1=(f_n\cdot F^n_1)\in \mathcal{A}^*\cdot B^*$. Note that $q^n_{\ell}(F^n_1)\rightarrow_n 0$ ($\ell\in\mathbb{N}$). In fact, $\sum_{\ell,n\in\mathbb{N}}q^n_{\ell}(F^n_1)<\infty$ and by Lemma \ref{plllqlllfunde} (ii) and (iv), we have
\begin{eqnarray*}
q_{M_{(n,\ell_n)}}(f_n\cdot F^n_1)&=&\sup_{x\in M_{(n,\ell_n)}}|f_n\cdot F^n_1(x)| 
=\sup_{x\in M_{(n,\ell_n)}}|F^n_1(x\cdot f_n)|  \\
&\leq&\sup_{x\in M_{(n,\ell_n)}}q^n_{\ell}(F^n_1)q_{M_{(n,\ell_n)}}(x\cdot f_n)
\rightarrow_n 0.
\end{eqnarray*}
Hence, $\mathcal{A}^*\cdot B^*\subseteq B$ and similarly, $B^*\cdot \mathcal{A}^*\subseteq B$. Therefore,
\begin{equation}\label{F22}
\mathcal{A}^*\cdot B^*\cup B^*\cdot \mathcal{A}^*\subseteq B.
\end{equation}
Now, for every $F,G\in\mathcal{A}^{**}$, by applying (\ref{F11}) and (\ref{F22}), 
$F\square G= F_1\square G_1$ and $F\Diamond G=F_1\Diamond G_1$,
where $F_1,G_1\in B^*$.
Since $B^*\cong\ell^1\text{-}\prod_{n\in\mathbb{N}}A^{**}_n$, we have $F_1\square G_1:=(F^n_1\square_n G^n_1)$ and $F_1\Diamond G_1:=(F^n_1\Diamond_n G^n_1)$. By Arens regularity of each $\mathcal{A}_n$, we have $F^n_1\square_n G^n_1=F^n_1\Diamond_n G^n_1$ ($n\in\mathbb{N}$), and so $ F_1\square G_1=F_1\Diamond G_1$. Thus, $\mathcal{A}$ is Arens regular.

The converse is obvious, since each $\mathcal{A}_n$ is a closed subalgebra of $\mathcal{A}$.
\end{proof}

The following example shows that in general $\ell^{\infty}\text{-}\prod_{n\in\mathbb{N}}\mathcal{A}_n$ is not Arens regular.

\begin{example}\label{example2}
Consider Banach algebras $\mathcal{A}_n=\ell^1(\mathbb{Z},w_{\frac{1}{n}})$ $(n\in\mathbb{N})$ where $$w_{\alpha}(t)=(1+|t|)^{\alpha}\;\;\;\;\;\;\;\;\;(t\in\mathbb{Z})$$ for each $\alpha>0$. Let
\begin{eqnarray*}
&&\mathcal{A}=c_0\text{-}\prod_{n\in\mathbb{N}}\mathcal{A}_n=\big{\{}(a_n)\in\prod_{n\in\mathbb{N}}\mathcal{A}_n:\|a_n\|_n\rightarrow0\big{\}}\;\;\;\;\;\text{and} \\
&&\mathcal{B}=\ell^{\infty}\text{-}\prod_{n\in\mathbb{N}}\mathcal{A}_n=\big{\{}(a_n)\in\prod_{n\in\mathbb{N}}\mathcal{A}_n:\sup_{n\in\mathbb{N}}\|a_n\|_n<\infty\big{\}}.
\end{eqnarray*}
The following statements hold.
\begin{enumerate}
\item[(i)]
by \cite[page 35]{Dales}, $\mathcal{A}$ is Arens regular.
\item[(ii)]
by \cite[Example 9.2]{Dales}, $\mathcal{B}$ is not Arens regular.
\item[(iii)]
$\mathcal{B}$ is a closed sualgebra of $\mathcal{A}^{**}$ and consequently $\mathcal{A}^{**}$ is not Arens regular.
\end{enumerate}
\end{example}

\vspace{9mm}

{\footnotesize \noindent
 Z. Alimohammadi\\
  Department of Mathematics,
   University of Isfahan,
    Isfahan, Iran\\
    z.alimohamadi62@yahoo.com\\

\noindent
 A. Rejali\\
  Department of Mathematics,
   University of Isfahan,
    Isfahan, Iran\\
    rejali@sci.ui.ac.ir\\


\footnotesize

\begin{thebibliography}{9}

\bibitem{Abtahi}
 F. Abtahi, B. Khodsiani, and A. Rejali, Arens regularity of inverse semigroup algebras, Bull. Iranian Math. Soc., 40 (2014), no. 6, 1527--1538.

\bibitem{Arikan}
N. Arikan, Arens regularity and reflexivity, Quarterly J. Math. Oxford (2), 32 (1981), 383--388.

\bibitem{Baker}
J. W. Baker and A. Rejali, On the Arens regularity of weighted convolution algebras, J. London Math. Soc. (2), 40 (1989), no. 3, 535--546.

\bibitem{Dales}
H. G. Dales and A. T.-M. Lau, The second duals of Beurling algebras, Mem. Amer. Math. Soc., 177 (2005).

\bibitem{Duncan}
J. Duncan and S. A. R. Hosseiniun, The second dual of a Banach algebra, Proc. Roy. Soc. Edinburgh 84(A) (1979), 309--325.

\bibitem{Eberlein}
F. Eberlein, Abstract ergodic theorems and weak almost periodic functions, Transactions of the American Mathematical Society, 67 (1949),  217--240.

\bibitem{Goldmann}
H. Goldmann, Uniform Fr\'echet algebras, North-Holland Mathematics Studies, 162. North-Holand (Amesterdam-New York, 1990).

\bibitem{Grothendiek}
A. Grothendiek, Crit\`eres de compacit\'e dans les espaces fonctionnels g\'en\'eraux, American J. Math. 74 (1952), 168--186.

\bibitem{Gulick}
S. L. Gulick, The bidual of a locally multiplicatively-convex
algebra, Pacific J. Math., 17/1 (1966), 71--96.

\bibitem{Hel2} 
A. Ya. Helemskii, The homology of Banach and topological algebras, Moscow University Press (in Russian); English transl: Kluwer Academic Publishers (Dordrecht 1989).

\bibitem{Khodsiani}
B. Khodsiani, A. Rejali, and H. R. E. Vishki, Arens regularity of certain weighted semigroup algebras and countability, A. Semigroup Forum 92 (2016), 304--310.

\bibitem{Kothe2}
G. K\"othe, Topological Vector Spaces. (II), Springer-Verlag (New York, 1979).

\bibitem{Meise}
R. Meise and D. Vogt, Introduction to functional analysis, Oxford Science Publications, (1997).

\bibitem{Pym}
J. S. Pym, The convolution of functionals on spaces of bounded functions, Proc. London Math. Soc. (3), 15 (1965), 84--104.

\bibitem{Pym1}
J. S. Pym, Remarks on the second duals of Banach algebras, J. Nigerian Math. Soc., 2 (1983), 31--33.

\bibitem{Rejali}
A. Rejali, H. R. E. Vishki, Regularity and amenability of the second dual of weighted group algebras, Proyecciones 26 (2007), 259--267.

\bibitem{Sch}
H. H. Schaefer, Topological vector spaces, Third printing, Graduate text in mathematics 3, Springer-Verlag (New York, 1971).

\bibitem{Young}
N. J. Young, Semigroup algebras having regular multiplication, Studia Math. Soc., Providence, R. I. (1961).

\bibitem{Zivari}
A. Zivari-Kazempour, On the Arens product and approximate identity in locally convex algebras, Filomat, 30 (2016), no. 6, 1493--1496.


\end{thebibliography}
\end{document}